\documentclass[11pt]{article}
\usepackage{amsmath,amsthm,amsfonts,amssymb,amscd, amsxtra, mathrsfs}
\usepackage{url}
\usepackage{cancel} 
\usepackage{rotating}
\usepackage[margin=2.6 cm,nohead]{geometry}
\usepackage{color}

\newtheorem{theorem}{Theorem}
\newtheorem{lemma}[theorem]{Lemma}
\newtheorem{definition}{Definition}

\newtheorem{proposition}[theorem]{Proposition}

\newtheorem{remark}{Remark}

\newtheorem{algorithm}{Algorithm}
\begin{document}
\title{Broyden quasi-Newton secant-type method for solving \\constrained mixed generalized equations}

\author{%
		P. C. da Silva Junior \thanks{IME, Universidade Federal de Goias,  GO, 74001-970, Brazil. E-mail: \texttt{silvapaulo@discente.ufg.br}}%
		\and%
		O. P. Ferreira \thanks{IME, Universidade Federal de Goias,  GO, 74001-970, Brazil. E-mail: {\tt orizon@ufg.br}. This author was   supported in part by CNPq grant 304666/2021-1.}%
		\and%
		G. N. Silva\thanks{Departamento de Matem\'atica, Universidade Federal do Piau\'i, Teresina, PI,   Brazil.. E-mail: \texttt{gilson.silva@ufpi.edu.br}%
		}%
	}

\maketitle

\begin{abstract}
This paper presents a novel variant of the Broyden quasi-Newton secant-type method aimed at solving constrained  mixed generalized equations, which can include functions that are not necessarily differentiable. The proposed method integrates the classical secant approach with techniques inspired by the Conditional Gradient method to handle constraints effectively. We establish local convergence results by applying the contraction mapping principle. Specifically, under assumptions of Lipschitz continuity, a modified Broyden update for derivative approximation, and the metric regularity property, we show that the algorithm generates a well-defined sequence that converges locally at a Q-linear rate. 
\end{abstract}

\noindent
{\bf Keywords:}  Constrained generalized equations, Secant method, quasi-Newton method, Lispchitz continuity, Broyden, metric regularity property, local convergence.

\medskip
\noindent
{\bf AMS subject classification:}  65K15, 49M15,  90C30.

\section{Introduction}

Introduced by S.~M.~Robinson in the early 1970s, generalized equations have become a versatile tool for describing, analyzing, and solving a wide range of problems in a unified manner, and have since been the focus of extensive research. For a thorough exploration of generalized equations and their applications, see \cite{DontchevRockafellar2010Book, josephy1979, Robinson1972_2, Robinson1980, Robinson1983} and the references therein. In particular, Newton's method and its variations for generalized equations have been extensively studied due to their appealing convergence properties~\cite{DontchevAragon2014, DontchevAragon2011, Bonnans, Dontchev1996, DontchevRockafellar2013,  FERREIRASilva2018, FerreiraSilva2017, RashidLiWu2013}. These studies have established conditions for superlinear and/or quadratic local convergence of Newton-type methods, often under assumptions such as strong regularity or metric regularity of the generalized equation, and Lipschitz-like conditions on the derivative of the vector-valued function. Moreover, the growing interest in developing both theoretical and computational tools for solving generalized equations stems from their ability to abstractly model various families of problems, including systems of nonlinear equations, equilibrium problems, linear and nonlinear complementarity problems, and variational inequality problems~\cite{DontchevRockafellar2010Book, FerrisPang1997, Robinson1979, Robinson1980, Robinson1982, Robinson1983}.

In this paper, we propose a method for solving generalized equations subject to a set of constraints, namely, to solve the following problem:
\begin{equation} \label{eq:GenEqI}
\mbox{find} \quad x \in C \quad \mbox{such that} \quad f(x) + g(x) + F(x) \ni 0,
\end{equation}
where $f: \Omega \to \mathbb{R}^n$ is a continuously differentiable function, $g: \Omega \to \mathbb{R}^n$ is a continuous function not necessarily differentiable, $\Omega \subseteq \mathbb{R}^n$ is an open set, $C \subset \Omega$ is a closed and convex set, and $F: \mathbb{R}^n \rightrightarrows \mathbb{R}^n$ is a set-valued mapping with closed nonempty graph. Constrained generalized equations of the form~\eqref{eq:GenEqI} encompass important mathematical problems, including the constrained variational inequality problem. For instance, if $g(x) = 0$ for all $x \in \Omega$ and $F$ is the normal cone mapping $N_D: \mathbb{R}^n \rightrightarrows \mathbb{R}^n$ of a convex set $D \subset \mathbb{R}^n$, then problem~\eqref{eq:GenEqI} reduces to the following constrained variational inequality problem:
\begin{equation} \label{VIP}
\mbox{find} \quad x \in C \quad \mbox{such that} \quad \langle f(x), y - x \rangle \geq 0, \qquad \forall~y \in D,
\end{equation}
where $C \subseteq D$ is a closed convex set and $\langle \cdot, \cdot \rangle$ denotes an  inner product in $\mathbb{R}^n$. A particular case of problem~\eqref{VIP} is the constrained nonlinear complementarity problem, which is formulated as
\[
\mbox{find} \quad x \in C \quad \mbox{such that} \quad x \in \mathbb{R}^n_+, \quad f(x) \in \mathbb{R}^n_+, \quad \text{and} \quad \langle f(x), x \rangle = 0,
\]
where $C \subseteq \mathbb{R}^n_+$. 
It is also known that if $F$ is the zero mapping, i.e., $F \equiv \{0\}$  and $g=0$, then problem~\eqref{eq:GenEqI} reduces to solving the constrained system of nonlinear equations $f(x) = 0$ with $x \in C$. This class of problems has been extensively studied, and various methods have been proposed for their solution; see, for example,~\cite{bellavia2006, MaxJefferson2017, GoncalvesOliveira2017, Cruz2014, MariniMoriniPorcelli2017}. Conversely, if $C = \Omega$ in problem~\eqref{eq:GenEqI}, we are dealing with unconstrained generalized equations, or simply generalized equations. 
An important example of this type is the Karush-Kuhn-Tucker (KKT) system for a standard nonlinear programming problem with a strict local minimizer; see~\cite[p.~269]{DontchevRockafellar2010Book},  for the case $g=0$.

Let $\mathbb{X}$ be a Banach space and $x_0, x_1 \in \mathbb{X}$ be given. Geffroy and Piétrus~\cite{GeoffroyPietrus2004} introduced the Newton-secant method for solving \eqref{eq:GenEqI} which consists in iteratively find $x_{k+1}$ as a solution of the  following subproblem, 
\begin{equation}\label{meth:geffroy2016}
    f(x_k) + g(x_k) + \big(f'(x_k) + [x_{k-1}, x_k; g]\big)(x_{k+1} - x_k) + F(x_{k+1}) \ni 0, 
\end{equation}  
where $k = 1, 2, \ldots$. The term $[x_{k-1}, x_k; g]$ represents the divided difference operator associated with $g:\mathbb{X}\to \mathbb{Y}$, $F: \mathbb{X} \rightrightarrows \mathbb{Y}$ is a set-valued mapping with closed nonempty graph, and $\mathbb{Y}$ is also a Banach space. It is known that the above iterative process was investigated by Catinas in~\cite{Catinas1994} for the specific case $F \equiv 0$. Hernández and Rubio~\cite{HernandezRubio2002} also performed a semilocal convergence analysis of this method \eqref{meth:geffroy2016} under the same conditions as in~\cite{Catinas1994}.  Additionally, significant modifications of the scheme \eqref{meth:geffroy2016} have been explored. For instance, Jean-Alexis and Piétrus~\cite{Jean-AlexisPietrus2008} proposed the following iterative approach for solving \eqref{eq:GenEqI}:  
\begin{equation*}
    f(x_k) + g(x_k) + \big(f'(x_k) + [2x_{k-1} - x_k, x_k; g]\big)(x_{k+1} - x_k) + F(x_{k+1}) \ni 0,
\end{equation*}  
for $k = 1, 2, \ldots$. Under appropriate assumptions, they established a local convergence analysis for this approach in~\cite{Jean-AlexisPietrus2008}, demonstrating a superlinear convergence rate. Several years later, Rashid, Wang, and Li~\cite{RashidWangLi2012} also conducted a local convergence analysis for the same method but under distinct assumptions compared to~\cite{Jean-AlexisPietrus2008}.  

It is widely recognized that Newton's method presents certain practical challenges. For instance, it requires the computation of the Jacobian matrix at each iteration, in addition to solving a linear system exactly. These requirements can render Newton's method inefficient, particularly for large-scale problems, as discussed in~\cite{leong2011matrix}.  To address these limitations, It was proposed in \cite{amaral2024quasi,de2024modified} the following quasi-Newton method for solving the unconstrained version of \eqref{eq:GenEqI} in Banach spaces, where the Fréchet derivative $f'(x_k)$ is replaced by a perturbed Fréchet derivative $A(x_k)$, which is computationally more efficient or simpler to evaluate:  
\begin{equation}\label{meth:vital1}
    f(x_k) + g(x_k) + A(x_k) + [x_{k-1}, x_k; g](x_{k+1} - x_k) + F(x_{k+1}) \ni 0, 
\end{equation}  
for $k = 1, 2, \ldots$. The function $A : \mathbb{X} \to \mathbb{Y}$ serves as an approximation of the Fréchet derivative $f' : \mathbb{X} \to L(\mathbb{X}, \mathbb{Y})$, and it satisfies a relaxed H\"older-type condition.
The main assumptions used in  \cite{amaral2024quasi} are:  (1) $f $ is Fréchet differentiable with a H\"older continuous derivative; (2)  $g$ is a continuous function allowing for the computation of first- and second-order divided differences; (3) $(f(x_1) + (A(x_1) + [x_0, x_1; g])(\cdot - x_1) + g(x_1) + F(\cdot))^{-1}$ is Aubin continuous at $0$ for $x_2$.  In  \cite{amaral2024quasi,de2024modified}, stronger assumptios are assumed to obtain linear convergence of the sequence $\{x_k\}$ generated by  \eqref{meth:vital1}.

The Broyden quasi-Newton secant-type method for generalized equations, originating from the works of N.~H.~Josephy~\cite{josephy1979}, has the following formulation: for the current iterates $x_{k-1}, x_k \in \mathbb{R}^n$, the next iterate $x_{k+1}$ is computed as a point satisfying the following inclusion:
\begin{equation} \label{eq:NewtonMethGenEq}
f(x_k) + g(x_k) + \big( B_k + [x_{k-1}, x_k; g] \big)(y_k - x_k) + F(y_k) \ni 0, \qquad k = 0, 1, \ldots,
\end{equation}
where $(B_k)_{k\in \mathbb{N}}$ is a sequence of bounded linear mappings between $\mathbb{R}^n$ satisfying the classical Broyden update rule, and $[x_{k-1}, x_k; g]$ is the first-order divided difference of the function $g: \Omega \to \mathbb{R}^n$ at the points $x_{k-1}$ and $x_k$. The specific choice of how $B_k$ is constructed dictates the particular quasi-Newton method being used, such as Broyden, BFGS, SR1, and others. Typically, the initial mapping $B_0$ is chosen to be a good approximation of $f'(\bar{x})$. Note that at each iteration, a partially linearized inclusion at the current iterate needs to be solved. The method described by~\eqref{eq:NewtonMethGenEq} thus serves as a framework for various iterative procedures in numerical nonlinear programming. For example, when $F \equiv \{0\}$, this method reduces to the standard Newton's method for solving systems of nonlinear equations. If $F$ represents the product of the negative orthant in $\mathbb{R}^s$ with the origin in $\mathbb{R}^{m-s}$, i.e., $F = \mathbb{R}^s_{-} \times \{0\}^{m-s}$, then~\eqref{eq:NewtonMethGenEq} corresponds to Newton's method for solving systems of nonlinear equalities and inequalities~\cite{Daniel1973}. Furthermore, if problem~\eqref{eq:GenEqI} with $C = \Omega$ represents the Karush-Kuhn-Tucker optimality conditions for a nonlinear programming problem, then~\eqref{eq:NewtonMethGenEq} describes the well-known sequential quadratic programming method~\cite[p.~384]{DontchevRockafellar2010Book}; see also~\cite{DontchevAs1996, IzmailovSolodov2010}. For unconstrained problems, where $C = \mathbb{R}^n$ or $C = \Omega$, we can simply set $x_{k+1} = y_k$. The iterative scheme described in~\eqref{eq:NewtonMethGenEq} was initially proposed by Catinas~\cite{Catinas1994} for the specific case where $F \equiv 0$ and $C = \mathbb{R}^n$, along with a local convergence analysis. This was followed by a semi-local analysis in~\cite{HernandezRubio2002}. Geoffroy and Pietrus~\cite{GeoffroyPietrus2004} extended this method to the case where $C = \mathbb{R}^n$ and $F \neq 0$, providing a local convergence analysis, with a semi-local analysis appearing subsequently in~\cite{HiloutPietrus2006}. Other significant variants of this method have been explored in works such as~\cite{daSilvaJunior2024, Jean-AlexisPietrus2008, RashidWangLi2012}.

Since the iterate $y_k$ in~\eqref{eq:NewtonMethGenEq} can sometimes be infeasible with respect to the constraint set $C$, a strategy is needed to ensure the feasibility of the generated sequence. To address this issue, in this paper,  we employ the conditional gradient (CondG) method, also known as the Frank-Wolfe method. Originally proposed for quadratic problems over a polyhedron~\cite{FrankWolfe1956}, this method has since been generalized to other problems~\cite{DemyanovRubinov1973, Dunn1979, Dunn1980, DunnHarshbarger1978, LEVITIN1966}. The CondG method has garnered significant attention in recent years~\cite{Nemirovski2015, LanZhou2016, RonnyTeboulle2013}, motivating its use in conjunction with the Broyden quasi-Newton secant-type method for solving problem~\eqref{eq:GenEqI}, as used, for example, in~\cite{daSilvaJunior2024}. From a theoretical perspective, we establish that the Broyden quasi-Newton secant-type method~\eqref{eq:NewtonMethGenEq} is well-defined and locally convergent, with a q-linear convergence rate. The key assumptions required for these results include metric regularity, Lipschitz continuity of the derivative $f'$, the approximation of $f'(x_k)$ via the sequence $\{ B_k \}$ of bounded linear mappings satisfying the classical Broyden update rule, and boundedness of the second-order divided difference of $g$. For a detailed discussion on metric regularity, see~\cite{DontchevRockafellar2013}.

The remainder of this paper is organized as follows. In Section~\ref{sec:int.1}, we present the notation and some technical results that are used throughout the paper. In Section~\ref{Sec:LocalAnalysis}, we describe the Broyden quasi-Newton secant-type method and present its local convergence properties. We conclude the paper with some remarks in Section~\ref{Sec:Conclusions}.

\section{Preliminaries} \label{sec:int.1}
In this section, we introduce the notation and basic concepts that will be used throughout the paper. Let $\|\cdot\|$ denote the Euclidean norm in ${\mathbb R}^n$. For $x \in \mathbb{R}^n$ and $\delta > 0$, we define the \emph{open ball} centered at $x$ with radius $\delta$ as
\(
B_{\delta}(x) := \{ y \in \mathbb{R}^n : \|x - y\| < \delta \},
\)
and the \emph{closed ball} as
\(
B_{\delta}[x] := \{ y \in \mathbb{R}^n : \|x - y\| \leq \delta \}.
\)
Let ${\mathscr L}({\mathbb R}^n, {\mathbb R}^n)$ denote the space of all continuous linear mappings from ${\mathbb R}^n$ to ${\mathbb R}^n$. For $A \in {\mathscr L}({\mathbb R}^n, {\mathbb R}^n)$, the norm of $A$ is defined by
\[
\|A\| := \sup\{ \|A x\| : \|x\| \leq 1 \}.
\]
Let $\Omega \subseteq \mathbb{R}^n$ be an open set. The derivative of a differentiable function $h: \Omega \to \mathbb{R}^n$ at $x \in \Omega$ is the linear mapping $h'(x): \mathbb{R}^n \to \mathbb{R}^n$.

\subsection{Divided Differences}
Divided differences are fundamental tools in numerical analysis and approximation theory, particularly in the context of interpolation and the analysis of iterative methods such as the secant method. In the study of quasi-Newton methods for solving nonlinear equations and generalized equations, divided differences are used to approximate derivatives when exact derivatives are unavailable or expensive to compute. In this subsection, we formalize the definitions of first and second order divided differences for vector-valued functions, which play a crucial role in the convergence analysis of our proposed method.

\begin{definition}[First Order Divided Difference]\label{def:divided_difference_first}
Let $\Omega \subseteq \mathbb{R}^n$ be an open set, and let $g: \Omega \to \mathbb{R}^n$ be a function. For any distinct points $x, y \in \Omega$ with $x \neq y$, the \emph{first order divided difference} of $g$ at the points $x$ and $y$ is the linear operator $[x, y; g] \in \mathscr{L}(\mathbb{R}^n, \mathbb{R}^n)$ defined by the relation
\begin{equation*}
[x, y; g](y - x) = g(y) - g(x).
\end{equation*}
Equivalently, we can write
\(
g(y) = g(x) + [x, y; g](y - x).
\)
When $g$ is differentiable at $x$, and in the limit as $y$ goes  to $x$, the first order divided difference converges to the derivative of $g$ at $x$. Therefore, we define
\begin{equation*}
[x, x; g] = g'(x),
\end{equation*}
where $g'(x)$ is the Jacobian matrix of $g$ at $x$.
\end{definition}

The first order divided difference $[x, y; g]$ can be viewed as an average of the derivatives of $g$ along the line segment from $x$ to $y$. It provides a linear approximation to $g$ over this segment, capturing the behavior of $g$ between $x$ and $y$. We proceed by introducing the second order divided difference. 
\begin{definition}[Second Order Divided Difference]\label{def:divided_difference_second}
Let $\Omega \subseteq \mathbb{R}^n$ be an open set, and let $g: \Omega \to \mathbb{R}^n$ be a function. For any distinct points $x, y, z \in \Omega$ with $x \neq y$, $x \neq z$, and $y \neq z$, the \emph{second order divided difference} of $g$ at the points $x$, $y$, and $z$ is the bilinear operator $[x, y, z; g] \in \mathscr{L}(\mathbb{R}^n, \mathscr{L}(\mathbb{R}^n, \mathbb{R}^n))$ defined by the relation
\begin{equation}\label{eq:SecOrderDividedDif}
[x, y, z; g](z - x) = [y, z; g] - [x, y; g].
\end{equation}
When $g$ is twice differentiable at $x$, and in the limit as $y$ and $z$ goes  to $x$, the second order divided difference converges to half the second derivative of $g$ at $x$. Therefore, we define
\begin{equation*}
[x, x, x; g] = \tfrac{1}{2} g''(x),
\end{equation*}
where $g''(x)$ denotes the Hessian matrix of $g$ at $x$.
\end{definition}
The second order divided difference $[x, y, z; g]$ captures the curvature of $g$ over the region spanned by $x$, $y$, and $z$. It provides a quadratic approximation to $g$ in this region, which is essential in understanding the behavior of iterative methods that utilize second-order information. In quasi-Newton and secant methods, divided differences serve as approximations to the Jacobian or Hessian matrices. Specifically, the first order divided difference $[x_k, x_{k-1}; g]$ approximates the derivative $g'(x_k)$ and is used to construct updates in the iterative scheme. The analysis of such methods relies on the properties of the divided differences and their relation to the true derivatives. The use of divided differences allows for derivative-free implementations of iterative methods, which can be advantageous when derivatives are difficult to compute. However, careful analysis is required to ensure that the approximations do not adversely affect the convergence properties of the method.

\subsection{Set-valued mappings}
Set-valued mappings generalize functions by allowing multiple outputs for each input. They are essential in the study of generalized equations.
\begin{definition}[Graph, Domain, and Range of a Set-Valued Mapping]
Let $F: \mathbb{R}^n \rightrightarrows \mathbb{R}^n$ be a set-valued mapping. The \emph{graph} of $F$ is defined as
\(
\mathrm{gph}\, F := \{ (x, y) \in \mathbb{R}^n \times \mathbb{R}^n : y \in F(x) \}.
\)
The \emph{domain} of $F$ is
\(
\mathrm{dom}\, F := \{ x \in \mathbb{R}^n : F(x) \neq \emptyset \},
\)
and the \emph{range} of $F$ is
\(
\mathrm{rge}\, F := \{ y \in \mathbb{R}^n : y \in F(x) \text{ for some } x \in \mathbb{R}^n \}.
\)
The \emph{inverse} of $F$ is the set-valued mapping $F^{-1}: \mathbb{R}^n \rightrightarrows \mathbb{R}^n$ defined by
\(
F^{-1}(y) := \{ x \in \mathbb{R}^n : y \in F(x) \}.
\)
\end{definition}

Understanding the distance between points and sets, as well as the excess of one set over another, is fundamental in variational analysis.

\begin{definition}[Distance and Excess]
Let $C$ and $D$ be subsets of $\mathbb{R}^n$. The \emph{distance} from a point $x \in \mathbb{R}^n$ to the set $D$ is defined by
\(
d(x, D) := \inf_{\hat{x} \in D} \| x - \hat{x} \|.
\)
The \emph{excess} of $C$ over $D$ is defined by
\(
e(C, D) := \sup_{x \in C} d(x, D).
\)
We adopt the conventions that $d(x, D) = +\infty$ when $D = \emptyset$, $e(\emptyset, D) = 0$ when $D \neq \emptyset$, and $e(\emptyset, \emptyset) = +\infty$.
\end{definition}

Metric regularity is a key concept in variational analysis and is instrumental in establishing convergence properties of iterative methods for solving generalized equations.

\begin{definition}[Metric Regularity]\label{def:metric_regularity}
Let $\Omega \subset \mathbb{R}^n$ be open and nonempty, and let $G: \Omega \rightrightarrows \mathbb{R}^m$ be a set-valued mapping. We say that $G$ is \emph{metrically regular} at $\bar{x} \in \Omega$ for $\bar{u} \in \mathbb{R}^m$ if $\bar{u} \in G(\bar{x})$, the graph of $G$ is locally closed at $(\bar{x}, \bar{u})$, and there exist constants  \textcolor{blue}{$\kappa > 0$}, $a > 0$, and $b > 0$ such that $B_a[\bar{x}] \subset \Omega$ and
\[
d(x, G^{-1}(u)) \leq \kappa d(u, G(x)), \quad \forall (x, u) \in B_a[\bar{x}] \times B_b[\bar{u}].
\]
If, in addition, the mapping $u \mapsto G^{-1}(u) \cap B_a[\bar{x}]$ is single-valued on $B_b[\bar{u}]$, we say that $G$ is \emph{strongly metrically regular} at $\bar{x}$ for $\bar{u}$.
\end{definition}

\begin{remark}
When the mapping $u \mapsto G^{-1}(u) \cap B_a[\bar{x}]$ is single-valued, we may write $w = G^{-1}(u) \cap B_a[\bar{x}]$ instead of $\{ w \} = G^{-1}(u) \cap B_a[\bar{x}]$ for simplicity.
\end{remark}

The contraction mapping principle is a fundamental result in fixed-point theory and will be used in our analysis.
\begin{theorem}[Contraction Mapping Principle]\label{thm:contraction_principle}
Let $\Phi: \mathbb{R}^n \rightrightarrows \mathbb{R}^n$ be a set-valued mapping, and let $\bar{x} \in \mathbb{R}^n$. Suppose there exist scalars $r > 0$ and $\zeta \in (0, 1)$ such that the set $\mathrm{gph}\, \Phi \cap (B_r[\bar{x}] \times B_r[\bar{x}])$ is closed, and
\begin{itemize}
    \item[(a)] $d(\bar{x}, \Phi(\bar{x})) \leq r(1 - \zeta)$;
    \item[(b)] $e(\Phi(x) \cap B_r[\bar{x}], \Phi(x')) \leq \zeta \| x - x' \|$, for all $x, x' \in B_r[\bar{x}]$.
\end{itemize}
Then, $\Phi$ has a fixed point in $B_r[\bar{x}]$; that is, there exists $\hat{x} \in B_r[\bar{x}]$ such that $\hat{x} \in \Phi(\hat{x})$. Moreover, if $\Phi$ is single-valued, then $\hat{x}$ is the unique fixed point of $\Phi$ in $B_r[\bar{x}]$.
\end{theorem}

\begin{proof}
See \cite[Theorem~5E.2, p.~313]{DontchevRockafellar2010Book}.
\end{proof}

\section{Broyden Quasi-Newton Secant-Type Method} \label{Sec:LocalAnalysis}
In this section, we introduce a {\it secant-type method} designed to solve the constrained problem~\eqref{eq:GenEq}. Assuming the metric regularity of the partial linearization of $f + g + F$, Lipschitz continuity of the derivative $f'$, and boundedness of the second-order divided difference of $g$, we analyze the sequence generated by the proposed inexact Broyden quasi-Newton secant-type method. Our method incorporates an inner subroutine, referred to as the {\it CondG procedure}, which computes a feasible inexact projection onto the set~$C$. This procedure is based on methods used in previous works (see, e.g., \cite{LanZhou2016}, \cite[p.~215]{Bertsekas1999}, \cite{OliveiraFerreiraSilva2019, MaxJefferson2017}). Before detailing the CondG procedure, we recall the concept of feasible inexact projection as introduced in~\cite{OliveiraFerreiraSilva2019}.

\begin{definition} \label{def:InexactProj}
Let $C \subset \mathbb{R}^n$ be a closed and convex set, $x \in C$, and $\theta \geq 0$. The \emph{feasible inexact projection mapping} relative to $x$ and forcing term $\theta$, denoted by $P_C(\cdot, x, \theta): \mathbb{R}^n \to C$, is the set-valued mapping defined as
\[
P_C(y, x, \theta) := \left\{ w \in C : \left\langle y - w, z - w \right\rangle \leq \theta \|y - x\|^2, \quad \forall z \in C \right\}.
\]
Each point $w \in P_C(y, x, \theta)$ is called a \emph{feasible inexact projection of $y$ onto $C$ with respect to $x$ and error tolerance $\theta$}.
\end{definition}

\begin{remark}
Since $C \subset \mathbb{R}^n$ is a closed and convex set, Proposition~2.1.3 of~\cite[p.~201]{Bertsekas1999} implies that for each $y \in \mathbb{R}^n$, the exact projection $P_C(y)$ belongs to $P_C(y, x, \theta)$. Therefore, $P_C(y, x, \theta) \neq \varnothing$ for all $y \in \mathbb{R}^n$ and $x \in C$. Moreover, when $\theta = 0$, we have $P_C(y, x, 0) = \{ P_C(y) \}$ for all $y \in \mathbb{R}^n$ and $x \in C$. In this case, we simply write $P_C(y, x, 0) = P_C(y)$.
\end{remark}

The following lemma, whose proof can be found in~\cite[Lemma~4]{MaxJefferson2017} (see also~\cite{OliveiraFerreiraSilva2019}), establishes a basic property of the feasible inexact projection and plays a crucial role in our analysis.

\begin{lemma} \label{pr:condi}
Let $y, \tilde{y} \in \mathbb{R}^n$, $x, \tilde{x} \in C$, and $\theta \geq 0$. Then, for any $u \in P_C(y, x, \theta)$, we have
\[
\left\| u - P_C(\tilde{y}, \tilde{x}, 0) \right\| \leq \| y - \tilde{y} \| + \sqrt{2\theta} \| y - x \|.
\]
\end{lemma}

We now present the CondG procedure, which computes a feasible inexact projection onto the set~$C$. This procedure is based on the conditional gradient method, also known as the Frank-Wolfe algorithm (see, e.g., \cite{Bertsekas1999}). The CondG procedure with input data $v$, $u$, and $\theta$ is formally described as follows.

\begin{algorithm}   \label{Alg:CondG}
\begin{footnotesize}
{\bf (CondG Procedure)}
\begin{description}
\item[\textbf{Input}] Choose $\theta \geq 0$, $v \in \mathbb{R}^n$, and $w \in C$ (with $w \neq v$). Set $w_0 = w$ and $\ell = 0$.
\item[\textbf{Step 1}]  Use a Linear Optimization  oracle to compute an optimal solution $z_\ell$ and the optimal value $s_{\ell}^*$ as
\begin{equation} \label{eq:CondG_C}
z_\ell := \arg\min_{z \in C} \langle w_\ell - v, z - w_\ell \rangle, \qquad s_\ell^* := \langle w_\ell - v, z_\ell - w_\ell \rangle.
\end{equation}
\item[\textbf{Step 2}] If $-s_\ell^* \leq \theta \| v - u \|^2$, then \textbf{stop} and set $w^+ = w_\ell$.
\item[\textbf{Step 3}] Compute $\alpha_\ell \in \, (0,1]$ and update  $w_{\ell+1}$ as 
\[
\alpha_\ell := \min \left\{ 1, \frac{ -s_\ell^* }{ \| z_\ell - w_\ell \|^2 } \right\},\qquad \qquad 
w_{\ell+1} := w_\ell + \alpha_\ell ( z_\ell - w_\ell ).
\]
\item[\textbf{Step 4}] Set $\ell \gets \ell + 1$ and return to \textbf{Step 1}.
\item[\textbf{Output}] $w^+ = w_\ell$.
\end{description}
\end{footnotesize}
\end{algorithm}

We now highlight the key aspects of Algorithm~\ref{Alg:CondG}. The algorithm aims to approximate the projection of $v$ onto the set  $C$ by iteratively improving the estimate $w_\ell$. At each iteration, the algorithm solves the linearized subproblem~\eqref{eq:CondG_C}, which involves minimizing the linear approximation of the distance function over~$C$. The optimal value $s_\ell^*$ obtained in~\eqref{eq:CondG_C} measures the potential reduction in the objective function. If this reduction is smaller than the specified tolerance, i.e., $-s_\ell^* \leq \theta \| v - u \|^2$, the algorithm terminates, returning $w^+ = w_\ell$ as the approximate projection. If the condition is not met, the algorithm updates the current estimate by moving towards $z_\ell$ using a step size~$\alpha_\ell$. The choice of $\alpha_\ell$ ensures sufficient decrease in the objective function while keeping $w_{\ell+1}$ within~$C$ due to the convexity of~$C$. Under mild assumptions, it can be shown that the sequence $( s_\ell^* )_{\ell \in \mathbb{N}}$ converges to zero (see, e.g., \cite[Proposition~A.2]{BeckTeboulle2004}), guaranteeing that the algorithm will terminate after a finite number of iterations. In essence, Algorithm~\ref{Alg:CondG} provides an efficient method for computing an approximate projection onto~$C$ that satisfies the inexact projection condition required by our main algorithm.

Next, we present the inexact Broyden quasi-Newton secant-type method for solving the mixed generalized equation~\eqref{eq:GenEq}. The method utilizes the CondG procedure to handle constraints and employs a Broyden update to approximate the Jacobian of~$f$. The algorithm is formally described below.

\begin{algorithm} \label{Alg:BroydenQNS}
\begin{footnotesize}
{\bf(Inexact Broyden Quasi-Newton Secant-Type Method)}
\begin{description}
\item[\textbf{Step 0}] Initialize with $x_{-1}, x_0 \in C$, a sequence $(\theta_k)_{k \in \mathbb{N}} \subset [0, +\infty)$, and $B_0 \in \mathscr{L}(\mathbb{R}^n, \mathbb{R}^n)$. Set $k = 0$.
\item[\textbf{Step 1}] If $0 \in f(x_k) + g(x_k) + F(x_k)$, \textbf{stop}; a solution has been found. Otherwise, compute $y_k \in \mathbb{R}^n$ satisfying
\begin{equation} \label{eq:aa0}
0 \in f(x_k) + g(x_k) + \left( B_k + [x_{k-1}, x_k; g] \right) ( y_k - x_k ) + F(y_k ).
\end{equation}
Update
\[
s_k := y_k - x_k, \qquad  \quad z_k := f(y_k) - f(x_k),
\]
and compute
\begin{equation} \label{eq:BroydenUpdateA1}
B_{k+1} := B_k + \frac{ ( z_k - B_k s_k ) s_k^\top }{ \| s_k \|^2 }.
\end{equation}
\item[\textbf{Step 2}] If $y_k \in C$, set $x_{k+1} = y_k$. Otherwise, compute $x_{k+1} \in P_C( y_k, x_k, \theta_k )$ using the CondG procedure.
\item[\textbf{Step 3}] Set $k \gets k + 1$ and return to \textbf{Step 1}.
\end{description}
\end{footnotesize}
\end{algorithm}

At each iteration, the Algorithm~\ref{Alg:BroydenQNS} checks whether the current point $x_k$ satisfies the inclusion defining the solution. If not, it computes an auxiliary point $y_k$ by solving an inclusion that involves a linear approximation of $f$ and $g$ at $x_k$, along with the operator $F$ at $y_k$. The Broyden update~\eqref{eq:BroydenUpdateA1} updates the approximation $B_k$ of the Jacobian of~$f$, ensuring that it satisfies the secant condition $B_{k+1} s_k = z_k$. If $y_k$ lies within the constraint set~$C$, it becomes the next iterate. If not, we compute an approximate projection of $y_k$ onto~$C$ using the CondG procedure, ensuring that the next iterate $x_{k+1}$ satisfies the inexact projection condition with tolerance~$\theta_k$. By iteratively applying this process, the algorithm generates a sequence $(x_k)_{k \in \mathbb{N}}$ within~$C$, which under appropriate conditions converges to a solution of~\eqref{eq:GenEq}.

\begin{remark}
When $g = 0$ and $F \equiv \{0\}$, Algorithm~\ref{Alg:BroydenQNS} reduces to the method studied in~\cite{MaxJefferson2017}. Additionally, if the constraint set is the entire space, i.e., $C = \mathbb{R}^n$, the algorithm becomes equivalent to the finite-dimensional version of the secant method proposed in~\cite{GeoffroyPietrus2004}.
\end{remark}

In the next section, we proceed with the analysis of the sequence  $(x_k)_{k \in \mathbb{N}}$, which we assume to be infinite. Note that, in this case we have $x_k\neq x_*$ and  $x_{k+1} \neq x_k$ for all $k$.
\subsection{Convergence analysis}

In this section, we analyze the convergence properties of Algorithm~\ref{Alg:BroydenQNS}. To establish the convergence of the algorithm, we first present some preliminary results that will be instrumental in our analysis. We then proceed to state and prove our main convergence theorem. We begin by recalling a theorem on perturbed metric regularity, which will play a key role in our analysis its proof can be found in \cite{DontchevAragon2014}.

\begin{theorem}[Perturbed Metric Regularity] \label{th:PertMetricReg}
Let $H: \mathbb{R}^n \rightrightarrows \mathbb{R}^m$ be a set-valued mapping with a closed graph, and suppose that at $(\bar{x}, \bar{y}) \in \operatorname{gph} H$, the mapping $H$ is metrically regular with constants $a > 0$, $b > 0$, and $\kappa > 0$. Let $\mu > 0$ be such that $\kappa \mu < 1$, and let $\kappa' > \kappa$. Then, for every $\alpha > 0$ and $\beta > 0$ satisfying
\begin{equation} \label{various}
    \alpha \leq {a}/{2}, \qquad \mu \alpha + 2 \beta \leq b, \qquad \text{and} \quad 2\kappa' \beta \leq \alpha(1 - \kappa \mu),
\end{equation}
and for every function $h: \mathbb{R}^n \to \mathbb{R}^m$ satisfying
\begin{equation} \label{condh1}
    \|h(\bar{x})\| \leq \beta, \quad \text{and} \quad \|h(x) - h(x')\| \leq \mu \|x - x'\|, \quad \forall\, x, x' \in B_\alpha(\bar{x}),
\end{equation}
the mapping $h + H$ has the following property: for every $y, y' \in B_\beta(\bar{y})$ and every $x \in (h + H)^{-1}(y) \cap B_\alpha(\bar{x})$, there exists $x' \in (h + H)^{-1}(y')$ such that
\begin{equation} \label{metlip}
    \|x - x'\| \leq \dfrac{\kappa'}{1 - \kappa \mu} \|y - y'\|.
\end{equation}
Moreover, if $H$ is \emph{strongly metrically regular} at $\bar{x}$ for $\bar{y}$—that is, the mapping $y \mapsto H^{-1}(y) \cap B_\alpha[\bar{x}]$ is single-valued and Lipschitz continuous on $B_b[\bar{y}]$ with Lipschitz constant $\kappa$—then under the same conditions on $\mu$, $\kappa'$, $\alpha$, and $\beta$, and for any function $h$ satisfying \eqref{condh1}, the mapping $y \mapsto (h + H)^{-1}(y) \cap B_\alpha[\bar{x}]$ is Lipschitz continuous on $B_\beta(\bar{y})$ with Lipschitz constant $\dfrac{\kappa'}{1 - \kappa \mu}$.
\end{theorem}

Assuming that $h(\bar{x}) = 0$, Theorem~\ref{th:PertMetricReg} indicates that if $H$ is (strongly) metrically regular at $\bar{x}$ for $\bar{y}$ and $h$ has a sufficiently small Lipschitz constant, then the perturbed mapping $h + H$ is also (strongly) metrically regular at $\bar{x}$ for $\bar{y}$. Specifically, the estimate \eqref{metlip} implies that $(h + H)^{-1}$ possesses the Aubin property at $\bar{y}$ for $\bar{x}$, which is equivalent to the metric regularity of $h + H$ at $\bar{x}$ for $\bar{y}$. Consequently,  it  recover the (extended) Lyusternik-Graves theorem as stated in \cite[Theorem~5E.1]{DontchevRockafellar2013}. For the strong regularity aspect, is  obtained a version of Robinson's theorem, as detailed in \cite[Theorem~5F.1]{DontchevRockafellar2013}.

To proceed with the analysis of Algorithm~\ref{Alg:BroydenQNS}, we assume throughout our presentation that the following assumptions hold:

\begin{itemize}
    \item[{\bf (A0)}] The derivative $f'$ is Lipschitz continuous with constant $L > 0$ on $\Omega$, i.e.,
    \begin{equation} \label{eq:LipContf}
        \|f'(x) - f'(y)\| \leq L \|x - y\|, \quad \forall\, x, y \in \Omega.
    \end{equation}
    \item[{\bf (A1)}] There exists $M > 0$ such that
    \begin{equation} \label{eq:LipContg}
        \|[x, y, z; g]\| \leq M, \quad \forall\, x, y, z \in \Omega,
    \end{equation}
    where $[x, y, z; g]$ denotes the second-order divided difference of $g$.
    \item[{\bf (A2)}] The mapping $f + g + F$ is metrically regular at $x_{*}$ for $0$ with constants  \textcolor{blue}{$\kappa > 0$}, $a > 0$, and $b > 0$.
\end{itemize}

We begin by establishing some technical results that will be useful in our analysis.

\begin{lemma} \label{Le:taylor}
Let $f: \mathbb{R}^n \to \mathbb{R}^n$ be a continuously differentiable function such that $f'$ satisfies {\bf (A0)}. Let $r > 0$ and $\epsilon > 0$ be such that $2L r < \epsilon$. Then, for all $x, y \in B_{r}(x_*)$, the following inequalities hold:
\begin{equation} \label{eq:taylor1}
    \|f(y) - f(x) - f'(x)(y - x)\| \leq \dfrac{L}{2} \|y - x\|^2,
\end{equation}
and
\begin{equation} \label{eq:taylor2}
    \|f(y) - f(x) - f'(x_*)(y - x)\| \leq \epsilon \|y - x\|.
\end{equation}
\end{lemma}
\begin{proof}
Since $f'$ is Lipschitz continuous with constant $L$, by applying  Taylor's theorem with integral the equality  \eqref{eq:taylor1} follows. For prove \eqref{eq:taylor2}, we have
\begin{align*}
\|f(y) - f(x) - f'(x_*)(y - x)\| &\leq \|f(y) - f(x) - f'(x)(y - x)\| + \|[f'(x) - f'(x_*)](y - x)\| \\
&\leq \big(({L}{/2}) \|y - x\| + L \|x - x_*\| \big) \|y - x\| \\
&\leq \epsilon \|y - x\|,
\end{align*}
where the last inequality holds because $\|y - x\| \leq 2 r$ and $2 L r < \epsilon$.
\end{proof}

The following proposition shows that the operator defined in \eqref{eq:BroydenUpdateA1} satisfies a boundedness condition that will be useful for our analysis. Its proof can be found in \cite{DontchevAragon2014}.

\begin{proposition} \label{pro:bounddeterioration}
Let $f: \mathbb{R}^n \to \mathbb{R}^n$ be a continuously differentiable function such that $f'$ satisfies {\bf (A0)}. Assume that $x_* \in \Omega$. Given $B_k \in \mathscr{L}(\mathbb{R}^n, \mathbb{R}^n)$ and $x_k, y_k \in \Omega$ with $y_k \neq x_k$, the operator $B_{k+1}$ defined as in \eqref{eq:BroydenUpdateA1} satisfies
\begin{equation} \label{detcond}
    \|B_{k+1} - f'(x_*)\| \leq \|B_k - f'(x_*)\| + \dfrac{L}{2} \left( \|y_k - x_*\| + \|x_k - x_*\| \right).
\end{equation}
\end{proposition}

With these preliminary results in hand, we now state and prove our main convergence theorem for Algorithm~\ref{Alg:BroydenQNS}.

\begin{theorem} \label{thm:convergence}
Let $\Omega \subseteq \mathbb{R}^n$ be an open set, $C \subseteq \Omega$ be a closed convex set, and let $f: \Omega \to \mathbb{R}^n$ be a continuously differentiable function. Let $g: \Omega \to \mathbb{R}^n$ be a continuous function (not necessarily differentiable), and let $F: \mathbb{R}^n \rightrightarrows \mathbb{R}^n$ be a set-valued mapping with closed graph. Suppose that $x_*$ is a solution to the generalized equation
\begin{equation} \label{eq:GenEq}
 f(x) + g(x) + F(x)\ni 0.
\end{equation}
Assume that conditions {\bf (A0)}, {\bf (A1)}, and {\bf (A2)} hold. Take  the initial approximation $B_0$ satisfying 
\begin{equation} \label{eq:BundkapB0}
\delta:=\|B_0 - f'(x_*)\| < \dfrac{1}{2 \kappa}.
\end{equation} 
Let $(\theta_k)_{k \in \mathbb{N}} \subset [0, +\infty)$ be a sequence with $\hat{\theta} = \sup_{k \in \mathbb{N}} \theta_k < \dfrac{1}{2}$. Then, there exists $\tau > 0$ such that, for any initial points $x_{-1}, x_0 \in B_{\tau}[x_*] \cap C$, there exists a  sequence $(x_k)_{k \in \mathbb{N}} \subset B_{\tau}[x_*] \cap C$ generated by Algorithm~\ref{Alg:BroydenQNS} with starting points $x_{-1}, x_0$ converges $q$-linearly to $x_*$. If, in addition, $f + g + F$ is strongly metrically regular at $x_*$ for $0$, then the sequence $(x_k)_{k \in \mathbb{N}}$ starting from $x_{-1}, x_0$ is unique and converges $q$-linearly to $x_*$.
\end{theorem}
\begin{proof}   Let  $L>0$ satisfying {\bf (A0)}, $M>0$ satisfying  {\bf (A1)} and $\kappa \geq 0$, $a>0$, and $b>0$ satisfying {\bf (A2)}. Choose $\kappa' > \kappa$ and  $\gamma > 0$ such that
\[
\frac{\kappa' \delta}{1 - \kappa \delta} < 1, \qquad \quad  \frac{\kappa' \delta}{1 - \kappa \delta} < \gamma < 1.
\]
To proceed with our analysis it is convenient to note the  function $(0, +\infty) \ni t \mapsto  {\kappa' t}/{(1 - \kappa t)}$ is strictly increasing. Then, we can choose a sufficiently small  $\epsilon > 0$  that satisfies the following two conditions
\begin{equation} \label{eq:gamma_condition}
 \epsilon<\delta, \qquad \quad  \frac{\kappa'}{1 - \kappa (\delta + \epsilon)} (\delta + \epsilon) < \gamma.
\end{equation}
In addition, considering that $\delta>0$ satisfies   \eqref{eq:BundkapB0}, without loss of generality, we can also assume that  $a>0$ and $\tau >0$  satisfy   the following three  conditions 

\begin{equation} \label{eq:csepx1}
a<\frac{\epsilon}{2L}, \qquad  \qquad \kappa \left(\delta + \frac{(L/2)a}{1-\gamma} + M\tau  \right) < 1.
\end{equation}

\begin{equation} \label{eq:csepx2}
\dfrac{\kappa'}{1 - \kappa \left( \delta + \dfrac{(L/2)a}{1 - \gamma} + M\tau \right)}\left(\delta+ \epsilon  + \dfrac{(L/2)a}{1 - \gamma} + M \tau \right) < \gamma.
\end{equation}
It is also convenient  to define the constant 
\begin{equation} \label{eq:csepxmu}
\mu := \delta + \frac{(L/2) a}{1 - \gamma} + M \tau .
\end{equation}
Choose  $\alpha > 0$ and $\beta > 0$ satisfying  \eqref{various} in Theorem~\ref{th:PertMetricReg}, i.e., satisfying following inequalities
\begin{equation} \label{eq:csepxtrm}
\epsilon + 3\alpha M \leq \mu, \qquad 0< \alpha \leq \frac{a}{2}, \qquad \mu \alpha + 2 \beta \leq b, \qquad 2 \kappa' \beta \leq \alpha (1 - \kappa \mu). 
\end{equation}
Finally, taking into account that  $\beta > 0$ , we can again make $\tau>0$ sufficiently small such that 
\begin{equation} \label{eq:csepxmuf}
0< \tau< \frac{1}{M} \left[ \frac{(\gamma - \sqrt{2 \theta_0})(1 - \kappa \mu)}{(1 + \sqrt{2 \theta_0}) \kappa'} - (\epsilon + \delta) \right], \qquad \qquad \left( \epsilon + \delta + \dfrac{(L/2)a}{1 - \gamma} + M\tau\right) \tau\leq \beta.
\end{equation}
Before  continuing with our analysis let us note that the first two inequalities in \eqref{eq:csepxmuf} implies that 
\begin{equation} \label{eq:csepxmuc}
\dfrac{\kappa'}{1 - \kappa \mu}\Big( \epsilon +  \delta+\frac{(L/2) a}{1 - \gamma}+ M\tau\Big)\left( 1 + \sqrt{2\theta_0}\right) + \sqrt{2\theta_0} \leq \gamma.
\end{equation}
Consider 
  \begin{equation}  \label{eq:Auxh}
h(x) = f(x_*) + g(x_*) + \left[ f'(x_*) + [z, x_*; g] \right](x - x_*) - f(x) - g(x),
\end{equation}
Note that $h(x_*) = 0$ and,   for all  $ x, x', z \in B_{\alpha}[x_*]$,  we have 
\begin{align} 
\| h(x) - h(x') \| & = \| f(x') + g(x')  - f(x) - g(x) + (f'(x_*) + [z, x_*; g])(x - x') \| \notag\\
& \leq \| f(x') - f(x) - f'(x_*)(x' - x) \| + \| g(x') - g(x) - [z, x_*; g] (x' - x) \|. \label{eq:fternpf}
\end{align}
To bound the first term in \eqref{eq:fternpf} we note that  \eqref{eq:taylor2} in Lemma \ref{Le:taylor} implies that 
  \begin{equation}  \label{eq:finqep}
		\|f(x')  - f(x) - f(x_*) (x' - x)\| \leq  \epsilon \|x' - x\|. 
\end{equation}
Now, we are going to bound the second  term  \eqref{eq:fternpf}. For that, first note that 
 \begin{align*} 
		\|g(x') - g(x) - [z, x_*; g](x' - x) \| & = |[x, x'; g](x' - x) - [z, x_*; g](x' - x)\| \\
        & = \|\left( [x, x'; g] - [z, x_*; g]\right)(x' - x)\| \\
        & = \|\left( [x, x'; g] - [z, x; g] - \left(  [z, x_*; g] - [z, x; g] \right)\right) (x' - x)\| 
	\end{align*}
Considering that $ [z, x_*; g]- [z, x; g]= [z, x_*; g]-[x, z; g]$, the last inequality implies that 
 \begin{align*} 
\|g(x') - g(x) - [z, x_*; g](x' - x) \|  & = \|\left( [x, x'; g] - [z, x; g] - \left(  [z, x_*; g] - [x, z; g] \right)\right) (x' - x)\|.
 \end{align*}	
Thus,  by  using \eqref{eq:SecOrderDividedDif} and  \textbf{A1}, and taking into account that $ x, x', z \in B_{\alpha}[x_*]$,  we conclude that 
 \begin{align*} 
\|g(x') - g(x) - [z, x_*; g](x' - x) \|   & = \|\left( [z, x, x'; g](x' - z) - [x, z, x_*; g](x_* - x)\right) (x' - x)\| \\
        & \leq \left(\| [z, x, x'; g] \| \|x' - z\| +  \|[x, z, x_*; g]\| \|x_* - x\| \right) \|x' - x)\| \\
        & \leq 3\alpha M \|x' - x\|.
	\end{align*}	
Hence,  using \eqref{eq:finqep} and  the last inequality we obtain  from \eqref{eq:fternpf} and the first inequality in \eqref{eq:csepxtrm} that 
 \begin{equation} \label{eq: fcmth}
\| h(x) - h(x') \|  \leq   (\epsilon + 3\alpha M) \|x' - x\| \leq \mu \|x' - x\|.\qquad \forall x, x' \in B_{\alpha}[x_*].
\end{equation}
To simplify the notation define  the mapping
\begin{equation} \label{eq:H_definition}
G(x) = f(x_*) + g(x_*) + \left[ f'(x_*) + [z, x_*; g] \right](x - x_*) + F(x).
\end{equation}
It follows from \eqref{eq:Auxh} and \eqref{eq:H_definition}  that  $G=H+h$, where $H=f+g+F$. It follows from  {\bf (A2)} that $H$ is  is metrically regular at $x_{*}$ for $0$ with constants $\kappa > 0$, $a > 0$, and $b > 0$. Considering that $h(x_*) = 0$, using \eqref{eq:csepxtrm} and \eqref{eq: fcmth} we can apply  Theorem~\ref{th:PertMetricReg} with ${\bar x}$  we conclude that $G$ is metrically regular at $x_*$ for $0$ with constant $\kappa$. Therefore, there exist positive constants $a$ and $b$ such that
\[
d(x, G^{-1}(y)) \leq \kappa d(y, G(x)), \quad \forall x \in B_a[x_*], \quad \forall y \in B_b[0].
\]
Let $O := B_{\tau}[x_*]$, and choose any  points $x_0, z \in O \cap C$ with $x_0\neq  z$, $x_0 \neq x_*$ and $z \neq x_*$. Define
\begin{equation} \label{eq:h0_definition}
h_0(x) := f(x_0) + g(x_0) + (B_0 + [z, x_0; g])(x - x_0) - f(x_*) - g(x_*) - (f'(x_*) + [z, x_*; g])(x - x_*),
\end{equation}
and $w_0 := f(x_0) + g(x_0) - f(x_*) - g(x_*) + (B_0 + [z, x_0; g])(x_* - x_0)$. Then, we have 
\[
w_0 = h_0(x_*).
\]
We proceed to estimate $\|h_0(x_*)\|$. For that,  using \eqref{eq:h0_definition}, some manipulations show that 
\begin{multline}  \label{eq:h0_estimate}
\|h_0(x_*)\| \leq \|f(x_0) - f(x_*) - f'(x_*)(x_0 - x_*)\| + \| (B_0 - f'(x_*))(x_* - x_0) \| + \\\| g(x_0) - g(x_*) + [z, x_0; g](x_* - x_0) \|.
\end{multline}
Let us first bound the right hand side of the last inequality. For the first tern, using \eqref{eq:taylor2} in Lemma \ref{Le:taylor}, we have
\[
\| f(x_0) - f(x_*) - f'(x_*)(x_0 - x_*) \| \leq \epsilon \| x_0 - x_* \|.
\]
We bound the second tern, by using  definition of $\delta$ in  \eqref{eq:BundkapB0}, which gives 
\[
\| (B_0 - f'(x_*))(x_* - x_0) \| \leq \delta \| x_0 - x_* \|.
\]
For the term involving $g$, we use the definition of the second-order divided difference and assumption \textbf{(A1)}. Thus, we have 
\begin{align*}
\| g(x_0) - g(x_*) + [z, x_0; g](x_* - x_0) \| & = \| [x_0, x_*; g](x_* - x_0) - [z, x_0; g](x_* - x_0) \| \\
& = \| [z, x_0, x_*; g](x_* - z)(x_* - x_0) \| \\
& \leq M \| x_* - z \| \| x_0 - x_* \| \leq M \tau \| x_0 - x_* \|.
\end{align*}
Combining these estimates and taking into account that $x_0\in  O \cap C$,  the inequality \eqref{eq:h0_estimate} yields
\begin{equation}  \label{eq:auxmthz}
\| h_0(x_*) \| \leq \left( \epsilon + \delta + M \tau \right) \| x_0 - x_* \|\leq \left( \epsilon + \delta + M \tau \right) \tau \leq \left( \epsilon + \delta + \dfrac{(L/2)a}{1 - \gamma} + M\tau\right) \tau 
\end{equation}
Thus, it follows from \eqref{eq:csepxmuf} that 
\[
\| h_0(x_*) \| \leq \beta.
\]
Hence $w_0 = h_0(x_*) \in B_{\beta}[0]$. Similarly,  give $x, x' \in B_{\alpha}(x_*)$ and considering $z = x_{-1}$,  by using definition \eqref{eq:h0_definition},  some algebraic manipulations imply that 
\begin{align*}
\| h_0(x) - h_0(x') \| & = \| (B_0 + [z, x_0; g])(x - x') - (f'(x_*) + [z, x_*; g])(x - x') \| \\
& \leq \| (B_0 - f'(x_*))(x - x') \| + \| [z, x_0; g] - [z, x_*; g] \| \| x - x' \|.
\end{align*}
Thus, after some calculations and using Definition~\ref{def:divided_difference_second}  we conclude that 
\begin{align*}
 \| h_0(x) - h_0(x') \| & \leq \| (B_0 - f'(x_*))(x - x') \| + \| [z, x_0; g] - [z, x_*; g] \| \| x - x' \| \\
& = \| (B_0 - f'(x_*))(x - x') \| + \| [z, x_0; g] - [x_*, z; g] + [x_*, z; g] -[z, x_*; g] \| \| x - x' \| \\
& = \| (B_0 - f'(x_*))(x - x') \| + \| [x_*, z, x_0; g](x_0 - x_*) + [z, x_*, z; g](z-z)\| \| x - x' \| \\
& \leq \| (B_0 - f'(x_*))(x - x') \| + \| [x_*, z, x_0; g]\| \|x_0 - x_*\| \| x - x' \| 
\end{align*}
Since $ [z, x_*, z; g](z-z)=0$,  using properties of the norm we obtain  from the last inequality that 
\begin{align*}
 \| h_0(x) - h_0(x') \|  \leq \| (B_0 - f'(x_*))(x - x') \| + \| [x_*, z, x_0; g]\| \|x_0 - x_*\| \| x - x' \| 
\end{align*}
Finally, using \eqref{eq:BundkapB0}, assumption {\bf (A1)} and  definition of $\mu$ in \eqref{eq:csepxmu}, we  conclude that 
\begin{equation*}
\| h_0(x) - h_0(x') \|  \leq \delta \| x - x' \| + M \tau \| x - x' \|  \leq \mu \| x - x' \|, \qquad \forall x, x' \in B_{\alpha}(x_*).
\end{equation*}

By Theorem~\ref{th:PertMetricReg}, with ${\bar x}=x_*$, $h = h_0$, $H=G$, $y=w_0$ and $x=x_*$, there exists $y_0 \in (h + G)^{-1}(0)$ satisfying \eqref{eq:aa0}, for $k = 0$, and
\[
\| y_0 - x_* \| \leq \frac{\kappa'}{1 - \kappa \mu} \| h_0(x_*) \|.
\]
It follows from the first inequality in  \eqref{eq:auxmthz} and the last inequality, by taking into account \eqref{eq:csepx2} and \eqref{eq:csepxmu}, that 
\begin{equation} \label{eq:bgg}
\| y_0 - x_* \|\leq \frac{\kappa'}{1 - \kappa \mu} \left( \epsilon + \delta +  M\tau\right) \|x_0 - x_*\| \leq  \gamma \| x_0 - x_* \|.
\end{equation}
Therefore, considering thta  $\gamma < 1$, we  conclude that  $y_0 \in O$. Next, applying Lemma~\ref{pr:condi} with $u = x_1$, $y = y_0$, $x = x_0$, $\tilde{y} = x_*$, $\tilde{x} = x_*$, and $\theta = \theta_0$,  we obtain that 
\begin{align*}
\| x_1 - x_* \| & \leq \| y_0 - x_* \| + \sqrt{2 \theta_0} \| y_0 - x_0 \| \\
& \leq \| y_0 - x_* \| + \sqrt{2 \theta_0} ( \| y_0 - x_* \| + \| x_* - x_0 \| ) \\
& \leq (1 + \sqrt{2 \theta_0}) \| y_0 - x_* \| + \sqrt{2 \theta_0} \| x_0 - x_* \|.
\end{align*}
Using  inequality \eqref{eq:csepxmuc} and \eqref{eq:bgg} we conclude that 
\begin{equation*}
\| x_1 - x_* \| \leq \left[ \frac{\kappa'}{1 - \kappa \mu} (\epsilon + \delta + M \tau) (1 + \sqrt{2 \theta_0}) + \sqrt{2 \theta_0} \right] \| x_0 - x_* \|  \leq \gamma \| x_0 - x_* \|.
\end{equation*}
Therefore, since $\gamma < 1$, we obtain that $x_1 \in O \cap C$.

By induction, suppose that for some $n \geq 1$, we have points $x_k \in O \cap C$, $y_k \in O$, satisfying
\[
\| y_{k-1} - x_* \| \leq \gamma \| x_{k-1} - x_* \|, \quad \qquad  \| x_k - x_* \| \leq \gamma \| x_{k-1} - x_* \|, \quad \qquad  k = 1, \dots, n.
\]
Thus, it follows from Proposition \ref{pro:bounddeterioration} together with the two last inequalities, that 
\begin{align*}
\| B_n - f'(x_*) \| & \leq \| B_0 - f'(x_*) \| + \dfrac{L}{2} \Big(\|x_0 - x_*\| + \sum_{k = 1}^{n} ( \| x_k - x_* \|) + \sum_{k = 1}^{n} \| y_{k-1} - x_* \| )\Big) \\
& \leq \delta + \dfrac{L}{2}\|x_0 - x_*\| + L \sum_{k = 1}^{n} \gamma^k \| x_0 - x_* \|  \leq \delta + L \sum_{k = 0}^{n} \gamma^k \| x_0 - x_* \|  \leq \delta + \frac{L \| x_0 - x_* \|}{1 - \gamma}.
\end{align*}
Hence, using \eqref{eq:csepxtrm}, we conclude that  
$$
\| B_n - f'(x_*) \| \leq \delta + \frac{(L/2) a}{1 - \gamma}.
$$
Consider $z = x_{-1}$ and  define
\[
h_n(x) := f(x_n) + g(x_n) + (B_n + [z, x_n; g])(x - x_n) - f(x_*) - g(x_*) - (f'(x_*) + [z, x_*; g])(x - x_*),
\]
and
\[
w_n := f(x_n) + g(x_n) - f(x_*) - g(x_*) + (B_n + [z, x_n; g])(x_* - x_n) = h_n(x_*).
\]
Using similar estimates as before, to obtain that \eqref{eq:auxmthz},  we can also obtain that 
\[
\| w_n \| \leq \Big( \epsilon + \delta + \frac{(L/2) a}{1 - \gamma} + M \tau \Big) \| x_n - x_* \| \leq \beta.
\]
Hence, $\|h_n(x_*)\| = \|w_n\| \leq \beta$. In addition,  also we can  obtain that 
\begin{align*}
\| h_n(x) - h_n(x') \| & = \| (B_n + [z, x_n; g])(x - x') - (f'(x_*) + [z, x_*; g])(x - x') \| \\
& \leq \| (B_n - f'(x_*))(x - x') \| + \| [z, x_n; g] - [z, x_*; g] \| \| x - x' \| \\
& \leq \Big( \delta + \frac{(L/2) a}{1 - \gamma}\Big) \| x - x' \| + M \tau \| x - x' \| \\
& = \mu \| x - x' \|, \qquad \forall x, x' \in B_{\alpha}[x_*].
\end{align*}
By applying Theorem~\ref{th:PertMetricReg}, with ${\bar x}=x_*$, $h = h_n$, $H=G$, $y=w_n$ and $x=x_*$, there exists $y_n \in (h_n + G)^{-1}(0)$ satisfying \eqref{eq:aa0} for $k = n$, and
\begin{equation} \label{eq:bggn}
\| y_n - x_* \| \leq \frac{\kappa'}{1 - \kappa \mu} \| w_n \| \leq  \frac{\kappa'}{1 - \kappa \mu} \Big( \epsilon + \delta+\frac{(L/2) a}{1 - \gamma} +  M\tau\Big) \|x_n - x_*\|  \leq \gamma \| x_n - x_* \|.
\end{equation}
Again, considering that $\gamma < 1$, we  conclude that $y_n \in O$. Next, applying Lemma~\ref{pr:condi} with $u = x_{n+1}$, $y = y_n$, $x = x_n$, $\tilde{y} = x_*$, $\tilde{x} = x_*$, and $\theta = \theta_0$, and the inequality \eqref{eq:csepxmuc}, we obtain
\begin{align*}
\| x_{n+1} - x_* \| & \leq \| y_n - x_* \| + \sqrt{2 \theta_0} \| y_n - x_n \| \\
& \leq \| y_n - x_* \| + \sqrt{2 \theta_0} ( \| y_n - x_* \| + \| x_* - x_n \| ) \\
& \leq (1 + \sqrt{2 \theta_0}) \| y_n - x_* \| + \sqrt{2 \theta_0} \| x_n - x_* \|.
\end{align*}
Using  inequality \eqref{eq:csepxmuc} and \eqref{eq:bggn} we conclude that 
\begin{equation*}
\| x_{n+1} - x_* \|  \leq \Big( \frac{\kappa'}{1 - \kappa \mu} (\epsilon + \delta+\frac{(L/2) a}{1 - \gamma} +M \tau) (1 + \sqrt{2 \theta_0}) + \sqrt{2 \theta_0} \Big) \| x_n - x_* \| \leq \gamma \| x_n - x_* \|.
\end{equation*}
Hence,  the induction is completed.  Moreover, taking into account that  $\gamma < 1$,  we conclude that $x_{n+1} \in O \cap C$. Therefore, 
\[
  \| x_k - x_* \| \leq \gamma \| x_{k-1} - x_* \|, \quad \qquad  k = 1, 2, \ldots.
\]
Consequently,  the sequence $(x_k)_{k \in \mathbb{N}} \subset B_{\tau}[x_*] \cap C$ and  converges q-linearly to $x_*$.

If $H = f + g + F$ is strongly metrically regular at $x_*$ for $0$, then the mapping $G$ is strongly metrically regular, and by the uniqueness part of Theorem~\ref{th:PertMetricReg}, the sequence $(x_k)_{k \in \mathbb{N}} \subset B_{\tau}[x_*] \cap C$ is unique and converges q-linearly to $x_*$. This completes the proof.
\end{proof}

\section{Conclusions} \label{Sec:Conclusions}
The proposed Broyden quasi-Newton secant-type method extends existing techniques for solving constrained generalized equations by incorporating elements of the Conditional Gradient method. The method achieves local Q-linear convergence under standard assumptions, such as Lipschitz continuity and metric regularity.  Future work could explore extensions to global convergence and performance improvements for large-scale applications, as well as adaptations for specific problem classes, such as variational inequalities and complementarity problems.


\end{document}